\documentclass[10pt,a4paper]{amsart}

\theoremstyle{plain}
\newtheorem{theorem}{Theorem}[section]
\newtheorem{lemma}[theorem]{Lemma}
\newtheorem{proposition}[theorem]{Proposition}

\theoremstyle{definition}
\newtheorem{definition}[theorem]{Definition}

\theoremstyle{remark}
\newtheorem{remark}[theorem]{Remark}

\newtheorem{notations}[theorem]{Notations}

\newcommand{\fc}{\ensuremath{\mathcal{F}}}

\newcommand{\cc}{\ensuremath{\mathcal{C}}}

\newcommand{\vc}{\ensuremath{\mathcal{V}}}

\newcommand{\Q}{\mathbb{Q}}
\newcommand{\derRi}{R^h\pi_*\Q _X}
\newcommand{\Hom}{Hom_{D(Y)}(\derRi,R\pi_*\Q_X[ h ])}
\newcommand{\HomG}{Hom_{D(Y)}(R\pi_*\Q_X, \Q_Y)}

\newcommand{\Ps}{\mathbb{P}}

\newcommand{\bQ}{\mathbb{Q}}

\def\bin #1#2 {\left( \matrix { #1 \cr #2 \cr } \right) }

\begin{document}

\title[ On the  Gysin morphism for the Blow-up of an ordinary singularity]
{On the existence of a  Gysin morphism for the Blow-up of an ordinary singularity}

\author{Vincenzo Di Gennaro }
\address{Universit\`a di Roma \lq\lq Tor Vergata\rq\rq, Dipartimento di Matematica,
Via della Ricerca Scientifica, 00133 Roma, Italy.}
\email{digennar@axp.mat.uniroma2.it}

\author{Davide Franco }
\address{Universit\`a di Napoli
\lq\lq Federico II\rq\rq, Dipartimento di Matematica e
Applicazioni \lq\lq R. Caccioppoli\rq\rq, P.le Tecchio 80, 80125
Napoli, Italy.} \email{davide.franco@unina.it}

\abstract In this paper  we characterize  the Blowing-up maps of
ordinary singularities for which there exists a natural Gysin
morphism, i.e. a bivariant class $\theta \in
Hom_{D(Y)}(R\pi_*\mathbb Q_X, \mathbb Q_Y)$, compatible with
pullback and with restriction to the complement of the
singularity.



\endabstract
\maketitle

\begin{center}
{\it{To the memory of Sacha}}
\end{center}

\bigskip
\section{Introduction}

One of the problems that have most attracted Sacha Lascu's
interest throughout his mathematical career, is the projective
contractability of a smooth divisor inside a smooth projective
variety (see e.g. \cite{Sacha1}, \cite{Sacha2} and
\cite{FrancoLascu}). In this paper we are aimed at a topological
problem closely related to this.
\par \noindent
More specifically, a consequence of the main result of \cite{FrancoLascu} is that a smooth space  curve $C\subset  \Ps^3$
is contractable on a general surface of large degree $X$  containing $C$ iff $C$ is  $\bQ$-subcanonical. For such curves there exists a morphism
$\pi: X\longrightarrow Y$ contracting $C$ to the unique ordinary singular point of $Y$. The main result of this paper is that the morphism
$\pi: X\longrightarrow Y$ admits a Gysin map iff $C$ is rational, and this condition is equivalent to say that $Y$ is an \textit{homology manifold}, hence Poincar\'e Duality holds true on $Y$ (compare with \S 4).

Consider a projective variety $Y\subseteq \Ps ^N(\mathbb{C})$ with
an \textit{ordinary singularity} $\infty \in Y$, i.e. a
singularity whose projective tangent cone $G\subseteq \Ps ^{N-1}$
 is smooth and connected. Set
\begin{equation}
\label{Blowupintr} X:= Bl_{\infty}(Y)
\stackrel{\pi}{\longrightarrow} Y
\end{equation}
the Blow-up at $\infty$.  Of course, in this context we have a
surjection of Chow groups
$$\mathbb{A}_{\bullet}(X)\longrightarrow \mathbb{A}_{\bullet}(Y) \longrightarrow 0,$$
but  very seldom it happens that the push-forward  for rational
homology groups
$$\pi_*: H_{\bullet}(X; \mathbb Q)\longrightarrow H_{\bullet}(Y; \mathbb Q)$$
is surjective too. The surjectivity of push-forward $\pi_*$ is
closely related to the existence of some kind of Gysin map, i.e. a
\lq\lq wrong way\rq\rq
 morphism between rational homology groups
\begin{equation}
\label{Gysinhom}
 H_{\bullet}(Y; \mathbb Q)\longrightarrow H_{\bullet}(X; \mathbb Q)
\end{equation}
or, dually, between rational cohomology groups
\begin{equation}
\label{Gysincohom}
 H^{\bullet}(X; \mathbb Q)\longrightarrow H^{\bullet}(X; \mathbb Q).
\end{equation}
The existence of natural morphisms like (\ref{Gysinhom}) or
(\ref{Gysincohom}) had been extensively studied by Fulton and
MacPherson in \cite{FultonCF}, where it is introduced the concept
of \textit{bivariant theory}. These are \lq\lq simultaneous
generalizations of covariant group valued homology-like theories
and contravariant ring valued cohomology-like theories\rq\rq
(\cite[p. v]{FultonCF}). A bivariant theory from a category $\cc$
to abelian groups assigns to each morphism
$X\stackrel{f}{\rightarrow} Y$ in $\cc$ a (usually graded) group
$T(X\stackrel{f}{\rightarrow}Y)$. Such an assignment must satisfy
appropriate    axioms that ensure the existence of products,
pullbacks and pushforwards. When $\cc$ is the category of complex
algebraic varieties a \textit{Topological bivariant theory} can be
defined in such a way that \cite[\S 7]{FultonCF},  \cite{DeCM}
$$T^i(X\stackrel{f}{\rightarrow}Y):=Hom_{D(Y)}(Rf_*\mathbb Q_X , \mathbb Q_Y[i]),
$$
where  $X\stackrel{f}{\rightarrow} Y$ is a proper morphism of
algebraic varieties, $i\in\mathbb Z$, and $D(Y)$ is the
\textit{bounded derived category} of sheaves of $\mathbb Q$-vector
spaces on $Y$. An element
$$\theta \in Hom_{D(Y)}(Rf_*\mathbb Q_X , \mathbb Q_Y[i])$$ produces Gysin-like
morphisms
$$H_{\bullet}(Y; \mathbb Q)\longrightarrow H_{\bullet+i}(X; \mathbb Q),
\hskip3mm H^{\bullet}(X; \mathbb Q)\longrightarrow
H^{\bullet-i}(Y; \mathbb Q).$$ Such morphisms turn out to be
defined only for particular maps of algebraic varieties. In
particular, natural Gysin maps   are defined when the target $Y$
is smooth and when $X\stackrel{f}{\rightarrow} Y$ is either local
complete intersection or flat (\cite[\S 4.5]{Verdier}).
Unfortunately, our map (\ref{Blowupintr}) is neither local
complete intersection nor flat, and in general it does not admit a
Gysin morphism. Nevertheless, in particular cases (see \S 4) there
is a Gysin morphism satisfying some natural conditions such as
compatibility with pullback and restriction to the complement of
the singularity (compare with Definition \ref{naturality}).

In this paper we prove a characterization of desingularizations
like (\ref{Blowupintr}) admitting a natural Gysin morphism (see
Theorem \ref{main}). What it turns out is that there exists a
natural Gysin morphism if and only if we have a decomposition in
$D(Y)$:
\begin{equation}
\label{eqintro} R\,\pi_*\mathbb Q_X\simeq \bigoplus_{h\geq 0}
R^h\pi_*\mathbb Q_X[-h].
\end{equation}

Such a decomposition is very reminiscent of the  Decomposition
Theorem, as stated e.g. in \cite[Remark 1.6.2, (3)]{DeCM2}, and of
the Laray-Hirsch Theorem (compare with the proof of Lemma 2.5 in
\cite{DGFCCM2}). Indeed, in view of Proposition \ref{nfake} (5),
formula (\ref{eqintro}) could be proved as a consequence of the
Decomposition Theorem. Nevertheless, in order to prove
(\ref{eqintro}), we prefer to follow the (somewhat more direct)
approach of \S 3.

One may ask whether our result holds true with $\mathbb
Z$-coefficients. We have in mind to return on this question in a
future paper.

\bigskip

\section{Notations}

\vskip5mm

\begin{notations}
\begin{enumerate}
\item  For any algebraic variety $X$ we will denote by $\mathbb Q_X$ the
\textit{constant sheaf} on $X$, by $\vc_X$ the category of sheaves
of $\mathbb Q_X$-modules, and by $D(X)$ the \textit{bounded
derived category} of $\vc_X$.
\item
Consider a projective variety $Y\subseteq  \Ps ^N$ with an
\textit{ordinary singularity} $\infty \in Y$, i.e. a singularity
whose projective tangent cone $G\subseteq \Ps ^{N-1}$ is smooth.
Set
$$X:= Bl_{\infty}(Y) \stackrel{\pi}{\longrightarrow} Y$$
 the Blow-up at $\infty$.
Of course we have an inclusion $X \subseteq Bl_{\infty}(\Ps ^N)$
and $G\stackrel{i}{\hookrightarrow}X$ coincides with the
exceptional divisor. Furthermore, we set
$$
\begin{array}{ccccc}
U &\stackrel{j}{\rightarrow}  &  X \\
\stackrel{\text{id}}{}\updownarrow\,\,\,&  &\,\,\stackrel{}\downarrow{\pi}\\
U& \stackrel{k}{\rightarrow} &Y.\\
\end{array}
$$
where $U:=X-G=Y- \{\infty\}$.
\item Since the morphism $\pi: X\longrightarrow Y$ is proper,
for any sheaf $\fc \in \vc_X$ the \textit{direct image with proper support} $\pi_!\fc\in \vc_Y$
\cite[\S 2.6]{Iversen},  \cite[Definition 2.3.21]{Dimca2}
 coincides with the ordinary direct image $\pi_*\,\fc\in \vc_Y$.
\end{enumerate}
\end{notations}

\begin{remark}
\label{propersupport} By definition of direct image with proper
support (\cite[\S 2.6]{Iversen},  \cite[Definition
2.3.21]{Dimca2}), the sheaf $k_!\mathbb Q_U$ ($j_!\mathbb Q_U$
resp.) can be identified with the subsheaf of $\mathbb Q_Y$ (of
$\mathbb Q_X$ resp.) consisting of sections with support contained
in $U$.
\end{remark}

\vskip3mm
\begin{definition}
\label{naturality} We will say that a graded morphism
$$ \theta: H^{\bullet }(X; \mathbb Q)\longrightarrow H^{\bullet }(Y; \mathbb Q)$$
is  \textit{natural} if  the following conditions are satisfied:
\begin{enumerate}
\item the composite of $\theta $ with the pullback
$$ \theta \circ \pi ^*:H^{\bullet }(Y; \mathbb Q)\longrightarrow H^{\bullet }(Y; \mathbb Q)
$$
is the identity map;
\item $\theta$ is compatible with restrictions on $U$:
$$j^*= k^*\circ \theta: H^{\bullet }(X; \mathbb Q)\longrightarrow H^{\bullet }(U; \mathbb Q).$$
\end{enumerate}
\end{definition}

\vskip3mm
\begin{definition}
\label{Gysin} Consider a \textit{(topological) bivariant class}
\cite[\S 7]{FultonCF}, \cite{DeCM}
$$\theta \in Hom_{D(Y)}(R\pi_*\mathbb Q_X, \mathbb Q_Y).$$ By abuse of notations, we also denote by
$\theta $ the map induced by such a class on the cohomology groups \cite{FultonCF}, \cite{DeCM}:
$$ \theta: H^{\bullet }(X; \mathbb Q)\longrightarrow H^{\bullet }(Y; \mathbb Q).$$
According to Definition \ref{naturality}, we will say that
$\theta$ defines a \textit{natural Gysin map} if the last morphism
is natural.
\end{definition}

\begin{lemma}
\label{relative} Keep notations as above. Then  pullbacks give
isomorphisms
\begin{equation}
\label{isomorphism} H^h(Y;
\mathbb Q) \simeq H^h(Y, \{\infty\}; \mathbb Q)\simeq  H^h(X,G; \mathbb Q), \hskip3mm \forall h\geq 1.
\end{equation}
Assume additionally that there exists a natural morphism:
$$\theta: H^{\bullet }(X; \mathbb Q)\longrightarrow H^{\bullet }(Y; \mathbb Q).$$ Then the map
\begin{equation}
\label{directsum}
  (\theta, i^*): H^{\bullet}(X;\mathbb Q)\longrightarrow H^{\bullet }(Y;\mathbb Q)\oplus H^{\bullet }(G;\mathbb Q)
\end{equation}
is an  isomorphism of graded groups in degree $\geq 1$.
\end{lemma}
\begin{proof}
The isomorphism $H^h(Y, \{\infty\}; \Q)\simeq H^h(Y; \Q)$ follows
from the long exact sequence:
$$\dots \longrightarrow H^h(Y, \{\infty\}; \Q)\longrightarrow  H^h(Y; \Q) \longrightarrow
H^h( \{\infty\}; \Q)\longrightarrow \dots$$ As for the isomorphism
$$ \pi^*: H^h(Y, \{\infty\}; \Q)\simeq H^h(X,G; \Q)$$
we are going to give two different proofs.

\textit{Topological proof:} consider a small open neighborhood
$\infty\in B$ and set $T:=\pi ^{-1}(B)$. Of course $T$ is a
tubular neighborhood of $G$ in $X$ and we have $\partial B\simeq
\partial T$. Moreover, $\{\infty\}$ and $G$ are tautly imbedded in
$Y$ and $X$ \cite[p. 289]{Spanier}. Hence we have
$$H^h(Y, \{\infty\}; \Q)\simeq H^h(Y, B; \Q)\simeq H^h(Y- B, \partial B; \Q)
$$
$$\simeq H^h(X- T, \partial T; \Q)\simeq H^h(X, T; \Q)\simeq H^h(X,G; \Q).$$

\textit{Sheaf theoretic  proof:} by \cite[Theorem 12.1]{Br},
\cite[Remark 2.4.5, (ii)]{Dimca2}, and Remark \ref{propersupport},
we have
$$H^h(Y, \{\infty\}; \Q)\simeq H^h(Y,k_!\Q _U)= H^h(Y,\pi_!(j_!\Q _U))$$
$$\simeq H^h(Y,\pi_*(j_!\Q _U))\simeq H^h(X,j_!\Q _U)\simeq H^h(X, G; \Q). $$
In order to prove (\ref{directsum}), look at the following long
exact sequence:
$$\dots \longrightarrow H^h(Y; \Q)\simeq H^h(X,G; \Q)\longrightarrow  H^h(X; \Q)
\longrightarrow H^h( G; \Q)\longrightarrow \dots$$ By Definition
\ref{naturality} (1), the map $\pi^*:H^h(Y; \Q)\longrightarrow
H^h(X; \Q) $ is injective $\forall h\geq 1$, so we have:
$$0 \longrightarrow H^h(Y; \Q)\longrightarrow  H^h(X; \Q) \longrightarrow H^h( G; \Q)\longrightarrow 0,$$
and we are done.
\end{proof}

\vskip5mm

\section{The main result}

\vskip5mm

\begin{notations}
\label{injres}
\begin{enumerate}
\item Combining \cite[I, Theorem 6.2]{Iversen} with \cite[\S 7.3.2]{Voisin}, we see that
the natural morphism $\Q_Y \longrightarrow \pi_* \Q_X$ in $\vc_Y$  is induced by an element
$$\iota_0\in Hom_{D(Y)}(\Q_Y, R\pi_* \Q_X).$$
\item We denote by $K^{\bullet }$ an injective resolution of $\Q_Y$.
\item
We denote by $I^{\bullet }$ an injective resolution of $\Q_X$. By
\cite[II, Corollary 4.13]{Iversen},
$J^{\bullet}:=\pi_*I^{\bullet}$ can be identified as the derived
direct image $R\pi_*\Q_X $ in $D(Y)$. So, when $h\geq 1$, $\derRi
= H^h(J^{\bullet})$ is the skyscraper sheaf supported on $\infty$,
with stalk at $\infty$ given by $H^h( G; \Q)$. Furthermore, the
morphism $\iota_0$ defined in (1) can be seen as an element in
$[K^{\bullet}, J^{\bullet}]$.
\end{enumerate}
\end{notations}

We prove the following result which will be needed throughout
the paper (compare with \cite[Proposition 1.2]{Del68}).

\begin{proposition}
\label{Deligne} Assume that, for any $h\geq 1$, the identity map
$\derRi \longrightarrow \derRi$ lifts to a morphism $\iota_h \in
\Hom$. Then  we have an isomorphism in $D(Y)$:
$$\bigoplus _{h\geq 0}\iota_h  : \Q_Y + \sum_{h\geq 1}  \derRi[-h ] \longleftrightarrow R\pi_*\Q_X.$$
\end{proposition}
\begin{proof}
The above morphism is well defined by \ref{injres}, (1). Since $G$
is connected we have $\pi_*\mathbb Q_X=\mathbb Q_Y$, i.e.
$H^0(J^{\bullet})=H^0(K^{\bullet})$.
Furthermore, the complex $\derRi [h]$ is injective $\forall
h\geq1$,  because  $\derRi$ is a skyscraper sheaf. So  we have
$$\Q_Y + \sum_{h\geq 1}\derRi[ -h ]=K^{\bullet }+ \sum_{h\geq 1}\derRi[ -h ] $$
in $D(Y)$. We are done because  our hypothesis implies that
$$H^h\left(K^{\bullet }+\sum_{h\geq 1}\derRi[ -h ]\right)=H^h(J^{\bullet }) \hskip3mm \forall
h,
$$
therefore the map $\bigoplus _{h\geq 0}\iota_h$ is a quasi-isomorphism.
\end{proof}

\begin{remark}
By   \cite[Proposition 1.2]{Del68}, the hypothesis of Proposition
\ref{Deligne} is equivalent to the fact that, for any
cohomological functor $T$ from $\vc_Y$ to an abelian category, the
spectral sequence
$$ E^{pq}_2= T(R^q\pi_*\Q_X [p]) \Rightarrow RT(\pi_*\Q_X[p+q])$$
degenerates at $E^{pq}_2$.
 \end{remark}

\begin{proposition}
\label{mainprop} Assume that there exists a natural morphism $$
H^{\bullet }(X; \Q)\longrightarrow H^{\bullet }(Y; \Q).$$ Then the
hypothesis of Proposition \ref{Deligne} is satisfied, i.e. the
identity map $\derRi \rightarrow \derRi$ lifts to a morphism
$\iota_h \in \Hom$ for any $h\geq 1$.
\end{proposition}
\begin{proof}
Keep notations as in \ref{injres}. Set $\Gamma^{\bullet}:=
\Gamma(J^{\bullet})$ and denote by $d^h: \Gamma^h\to \Gamma^{h+1}$
the differential. Then we have $H^h(X; \Q) =
H^h(\Gamma^{\bullet})$, and by hypothesis any element of $H^h(G;
\Q)$ can be lifted to an element  $ \gamma \in Ker\,d^h$. We claim
that any  $\alpha\in  H^h(G, \Q)$ can be lifted to an element
$\beta\in Ker\,d^h\subset \Gamma (J^h)=\Gamma (I^h)$ which is
supported on $\infty$. Of course, to prove our claim amounts to
show that any  $\alpha\in  H^h(G;\Q)$ can be lifted to an element
$\beta\in Ker\,d^h\subset \Gamma (J^h)=\Gamma (I^h)$ such that
$\beta\mid_U=0\in \Gamma (J^h\mid_U)$. But $\gamma\mid_U$ projects
to a cohomology class living in $Im(H^h(X; \Q)\to H^h(U; \Q))$. By
(2) of Definition \ref{naturality}, we have
$$Im(H^h(X; \Q)\to H^h(U; \Q))\subseteq Im(H^h(Y; \Q)\to H^h(U; \Q)).
$$
By Lemma \ref{relative}, we find
$$Im(H^h(Y; \Q)\to H^h(U; \Q))= Im(H^h(X, G; \Q)\to H^h(U; \Q)).
$$
Since
$$H^h(Y,k_!\Q _U)\simeq H^h(Y, \infty; \Q)\simeq H^h(X, G; \Q)\simeq H^h(X,j_!\Q _U)$$
 (\cite[Theorem 12.1]{Br}, \cite[Remark 2.4.5, (ii)]{Dimca2}), Remark \ref{propersupport} implies that
 there exists $\delta_U\in \Gamma(J^{h-1}\mid_U)$ and
$\sigma \in \Gamma(J^h)$ supported in $U$
such that
$$\gamma\mid_U-d^{h-1}(\delta_U)=\sigma\mid_U.$$
Finally,  there exists $\delta\in \Gamma(J^{h-1})$ with $\delta\mid_U=\delta_U$, because  $J^{h-1}$ is injective
(hence flabby). We conclude that the section
$$\gamma- \sigma -d^{h-1}(\delta)\in \Gamma(J^h)
$$
is supported on $\infty$.
Our claim is proved because $\sigma +d^{h-1}(\delta)\in \Gamma(J^h)$ vanishes in $H^h(G; \Q)$.

To conclude the proof, fix a basis $\alpha_r\in H^h(G; \Q)$ and
lift any $\alpha_r$ to a $\beta_r\in Ker\, d^h\subseteq \Gamma
(J^h)=\Gamma (I^h)$ as in the claim. We get an isomorphism between
$H^h(G; \Q)$ and a subspace of $\Gamma (I^h)$ consisting of
sections supported on $\infty$. We are done because such an
isomorphism projects to a monomorphism of sheaves:
$$R^h\pi_*\Q_X \hookrightarrow Ker\,(J^h\to J^{h+1}).$$
\end{proof}

\bigskip

\begin{theorem}
\label{main} Keep notations as above. Then the following
properties are equivalent:
\begin{enumerate}
\item there exists a natural Gysin map $\theta \in \HomG$;
\item there exists a natural morphism of graded groups
$$  H^{\bullet }(X; \Q)\longrightarrow H^{\bullet }(Y; \Q);$$
\item we have an isomorphism in $D(Y)$:
$$\bigoplus _{h\geq 0}\iota_h  : \Q_Y + \sum_{h\geq 1}  \derRi[-h ] \longleftrightarrow R\,\pi_*\Q_X.$$
\end{enumerate}
\end{theorem}
\begin{proof}
(1) $\Rightarrow$ (2) by Definitions \ref{naturality} and
\ref{Gysin}. (2) $\Rightarrow$ (3) just combining Propositions
\ref{Deligne} and \ref{mainprop}.

(3) $\Rightarrow$ (1): If $\bigoplus _{h\geq 0}\iota_h $ is an
isomorphism in $D(Y)$, the projection on the first summand of
$\Q_Y + \sum_{h\geq 1}  \derRi[-h ]$ represents   a bivariant
class
$$\theta \in \HomG$$
such that
$$\theta  \circ \iota_0= id\in Hom_{D(Y)}(\Q_Y, \Q_Y),$$
and the first condition of Definition \ref{naturality} is satisfied.

As for (2) of Definition \ref{naturality}, since the sheaves $\derRi$ are supported on $\infty$, we have
$$\Q_Y\mid_U =((\theta  \circ \iota_0)\Q_Y)\mid_U=\theta(\Q_X)\mid_U=\Q_X\mid_U
$$
in $D(U)$, and we are done.
\end{proof}

\vskip5mm

\section{Examples}

Examples of Blow-up admitting a natural Gysin morphism are the
following:

\medskip
$\bullet$ any surface $Y$ with a node (cfr. \cite{DeCM3}, p. 127,
Remark 3.3.3, and \cite{McCrory}, p. 159, Example 2);

\medskip
$\bullet$ the  cone $Y\subseteq \mathbb P^{N}$ over a smooth
projective variety $M\subseteq \mathbb P^{N-1}$ of dimension $m$
such that ${H^{\bullet}(M)}\cong {H^{\bullet}(\mathbb P^m)}$.

\bigskip
This follows from the Proposition \ref{nfake} below, which gives
further characterizations of the existence of a natural Gysin
morphism. As for the equivalence of properties $(2)$-$(5)$  in
Proposition \ref{nfake}, we think they are certainly well-known.
However we briefly give the proof for lack of a suitable reference
(cfr. \cite{DeCM3}, p. 127, Remark 3.3.3). In the sequel we will
denote by $IH(Y)$ the intersection cohomology of $Y$ (see e.g.
\cite{Dimca2}, p. 154-159), by $IC^{\bullet}_Y$  the
\textit{intesection cohomology complex} of $Y$ (cfr.
\cite{Dimca2}, p.159)  and by $H_k^{BM}(U)$ the Borel-Moore
homology of $U$ . Furthermore, following \cite{Massey} we say that
a variety $Y$ of dimension $n$ is a $\bQ$-\textit{intersection
cohomology manifold} when $IC^{\bullet}_Y\simeq \bQ_Y[n]$, in
$D(Y)$. We refer to \cite{FultonYT}, Appendix B, for some
properties of Borel-Moore homology which we need in the proof. All
cohomology and homology groups are with $\mathbb Q$-coefficients.

\bigskip
\begin{proposition}\label{nfake}
Let $Y\subseteq \mathbb P^N$ be a projective irreducible variety
of complex dimension $m+1$, with a unique singular point $\infty
\in Y$, which is an ordinary singularity. Let $\pi:X\to Y$ be the
Blow-up at $\infty$, with smooth and connected exceptional divisor
$G$. The following properties are equivalent.

\smallskip
(1) There exists a natural morphism of graded groups $\theta:
H^{\bullet}(X)\to H^{\bullet}(Y)$.

\smallskip
(2) The duality morphism $H^{\bullet}(Y)\stackrel{\cdot\,\cap
[Y]}\longrightarrow H_{2(m+1)-\bullet}(Y)$ is an isomorphism (i.e.
$Y$ satisfies Poincar\'e Duality).

\smallskip
(3) ${H^{\bullet}(G)}\cong {H^{\bullet}(\mathbb P^m)}$.

\smallskip
(4) The natural map map ${H^{\bullet}(Y)}\to{IH^{\bullet}(Y)}$is
an isomorphism.

\smallskip
(5) $Y$ is $\bQ$-intersection cohomology manifold.
\end{proposition}

\begin{proof} First we prove that $(2)$ is equivalent to $(3)$.

Since the singular locus of $Y$ is finite, by \cite{McCrory} we
know that $Y$ satisfies Poincar\'e Duality if and only if $Y$ is a
homology manifold, i.e. if and only if
$H^h(Y,Y\backslash\{y\})\cong H^h(\mathbb R^{2(m+1)},\mathbb
R^{2(m+1)} \backslash\{0\})$ for any $y\in Y$. This condition is
certainly verified if $y$ is a regular point of $Y$. Choose a
small closed ball $D\subseteq \mathbb P^N$ around $\infty$ and set
$B:=D\cap Y$. By excision we have
$H^h(Y,Y\backslash\{\infty\})\cong H^h(B,B\backslash\{\infty\})$.
Recall that $B$ is homeomorphic to the cone over the link
$K:=\partial D\cap Y$ of the singularity $\infty\in Y$, with
vertex at $\infty$ (\cite{Dimca1}, p. 23). In particular $B$ is
contractible. Therefore, from the long exact sequence of
cohomology of the couple $(B,B\backslash\{\infty\})$, it follows
that $Y$ is a homology manifold if and only if
$B\backslash\{\infty\}$ has the same $\mathbb Q$-homology type as
a sphere $S^{2m+1}$. This in turn is equivalent to say that the
link $K$  has the same $\mathbb Q$-homology type as a sphere
$S^{2m+1}$, because $K$ is a deformation retract of
$B\backslash\{\infty\}$. On the other hand, via deformation to the
normal cone, we may identify $K$ with the link of the vertex of
the projective cone over the exceptional divisor $G \subseteq
\mathbb P^{N-1}$. Restricting the Hopf bundle $S^{2N-1}\to \mathbb
P^{N-1}$ to $G$, we obtain an $S^1$-bundle $K\to G$ inducing the
Thom-Gysin sequence (\cite{Spanier}, p.260)
$$
\dots\to H^h(G)\to H^h(K)\to H^{h-1}(G)\to H^{h+1}(G)\to
H^{h+1}(K)\to\dots
$$
And this sequence implies that $K$ has the same $\mathbb
Q$-homology type as a sphere $S^{2m+1}$ if and only if
${H^{\bullet}(G)}\cong {H^{\bullet}(\mathbb P^m)}$.

\medskip
Now we are going to prove that $(2)$ is equivalent to $(4)$.

First assume that property $(2)$ holds true. Since the singular
locus of $Y$ is finite, by (\cite{Dimca2}, p. 157) we already know
that $H^h(Y)=IH^h(Y)$ if $h>m+1$. Moreover we know that
$IH^{m+1}=\Im(H^{m+1}(Y)\to H^{m+1}(U))$, where
$U=Y\backslash\{\infty\}$. On the other hand
$H^{h}(U)=H_{2(m+1)-h}^{BM}(U)$ (\cite{FultonYT}, p. 217, (26)),
and from the natural exact sequence (\cite{FultonYT}, p. 219,
Lemma 3)
$$
\dots\to H_i(\{\infty\})\to H_i(Y)\to H_{i}^{BM}(U)\to \dots
$$
we see that $H_{2(m+1)-h}^{BM}(U)=H_{2(m+1)-h}(Y)$ for $h <
2(m+1)$. In particular for $h=m+1$ we have
$H^{m+1}(U)=H_{m+1}(Y)$, and therefore $IH^{m+1}=\Im(H^{m+1}(Y)\to
H_{m+1}(Y))=H_{m+1}(Y)=H^{m+1}(Y)$. Finally, when $h<m+1$ then we
have $IH^h(Y)=H^h(U)=H_{2(m+1)-h}(Y)=H^h(Y)$. Conversely assume
that property $(4)$ holds true. Since intersection cohomology
verifies Poincar\'e Duality (\cite{Dimca2}, p. 158), we have:
$$
H^h(Y)=IH^h(Y)=(IH^{2(m+1)-h}(Y))^{\vee}=(H^{2(m+1)-h}(Y))^{\vee}=H_{2(m+1)-h}(Y).
$$

\medskip
Next we prove that property $(2)$ implies $(1)$.

To this purpose, consider the following commutative natural
diagram:
$$
\begin{array}{ccccccc}
H^{h-1}(G)&\stackrel {}{\to}&H^{h}(X,G)&\stackrel{}{\to}
&H^{h}(X)&\stackrel {}{\to}&H^{h}(G)\\
\stackrel {}{}\uparrow & &\stackrel {}{}\Vert & & \stackrel {\pi^*}{}\uparrow& &\stackrel {}{}\uparrow \\
H^{h-1}(\infty)&\stackrel
{}{\to}&H^{h}(Y,\infty)&\stackrel{}{\to}&H^{h}(Y)
&\stackrel {}{\to}&H^{h}(\infty).\\
\end{array}
$$
Since ${H^{\bullet}(G)}\cong {H^{\bullet}(\mathbb P^m)}$, it
follows that the restriction map $H^{h-1}(X)\to H^{h-1}(G)$ is
surjective for any $h$. Therefore from previous diagram we deduce
that the pull-back $ \pi^*:H^{h}(Y)\rightarrow H^{h}(X)$ is
injective, and that $H^h(Y)\cong \pi^*H^h(Y)=H^h(X)$ if $h$ is odd
(either $h=0$ or $h\geq 2(m+1)$), and that $\dim H^h(X)-\dim
H^h(Y)= 1$ if $h$ is even with $2\leq h\leq 2m$. Next put:
$$
I^h:=\ker(H^h(X)\stackrel{j^*}\to H^h(U))=\Im(H^h(X,U)\to H^h(X)).
$$
Since $X$ is smooth we have $H^h(X,U)\cong H_{2(m+1)-h}(G)$
(\cite{Spanier}, p.351, Lemma 14), and the map $H^h(X,U)\to
H^h(X)$ identifies with the push-forward:
\begin{equation}\label{pd}
H_{2(m+1)-h}(G)\to H_{2(m+1)-h}(X)\cong H^{h}(X).
\end{equation}
Therefore $\dim I^h=1$ if $h$ is even with $2\leq h\leq 2(m+1)$,
and $\dim I^h=0$ otherwise. Now consider the following natural
diagram, where all maps are restrictions:
$$
\begin{array}{ccccc}
H^h(X)&\stackrel{\pi^*}{\hookleftarrow}  & H^h(Y)  \\
\stackrel{j^*}{}\downarrow\quad&\swarrow\stackrel{k^*}{}\\
H^h(U).\\
\end{array}
$$
By functoriality this diagram commutes:
$$
k^*=j^*\circ \pi^*.
$$
As before, when $h$ is even with $2\leq h\leq 2m$, the restriction
map $k^*:H^h(Y)\to H^h(U)$ identifies with the duality morphism
$H^h(Y)\to H_{2(m+1)-h}(Y)$, which is bijective by our assumption.
Since $I^h\cap \pi^*H^k(Y)=\pi^*(\ker k^*)$, it follows that
$$
I^h\cap \pi^*H^h(Y)=0.
$$
Therefore, by dimensional reasons, we have
$$
H^h(X)= \begin{cases}\pi^*H^h(Y)\oplus I^h \quad{\text{if $h$ is
even with $2\leq h\leq 2m$}}\\
\pi^*H^h(Y) \quad{\text{otherwise}}.
\end{cases}
$$
We are in position to define the natural morphism
$\theta:H^h(X)\to H^h(Y)$. In fact, when $h$ is even with $2\leq
h\leq 2m$, for any $x=\pi^*(y)+i\in H^h(X)=\pi^*H^h(Y)\oplus I^h$
put:
$$
\theta(x):=y,
$$
and put $$\theta:=(\pi^*)^{-1}$$ otherwise. It is evident that
$\theta\circ \pi^*={\rm{id}}_{H^{\bullet}(Y)}$, and that $k^*\circ
\theta=j^*$ when $\theta=(\pi^*)^{-1}$, because $j^*\circ
\pi^*=k^*$. When $h$ is even with $2\leq h\leq 2m$, taking into
account that $I^h=\ker j^*$, we have: $
(k^*\circ\theta)(x)=k^*(y)=(j^*\circ
\pi^*)(y)=j^*(\pi^*(y)+i)=j^*(x)$.

\medskip
Finally we are going to prove that property $(1)$ implies $(2)$.

Since $\theta\circ \pi^*={\text{id}}_{H^k(Y)}$ we get the
decomposition
\begin{equation}\label{dec}
H^h(P)=\pi^*H^h(Y)\oplus \,\ker\theta.
\end{equation}
Hence $\theta$ is compatible with restrictions on $U$ if and only
if
$$
\ker\theta\subseteq I^k
$$
(in fact for any $x\in H^h(X)$, $x=\pi^*y+c$, with
$c\in\ker\theta$, we have $(k^*\circ\theta) (x)=k^* (y)=(j^* \circ
\pi^*) (y)$; therefore $(k^*\circ\theta)(x)=j^* (x)$ if and only
if $j^* (x)=j^* (\pi^* y+c)=(j^*\circ \pi^*)(y)+j^* (c)=(j^*\circ
\pi^*) (y)$, hence if and only if $j^* (c)=0$). On the other hand,
combining Lemma \ref{relative} with (\ref{dec}), we get for $h>0$
$$
\dim \ker\theta=\dim H^h(G),
$$
and therefore (cfr. (\ref{pd})), for any $h>0$,
$$
\dim H^h(G)\leq \dim I^h\leq \dim H_{2(m+1)-h}(G)= \dim
H^{h-2}(G).
$$

Finally, for the equivalence between $(2)$ and $(5)$ we refer the reader to \cite{McCrory} and \cite{Massey}.
\end{proof}

\begin{remark}
\label{nfinalremark} From previous Proposition we see that,  if
existing, the natural morphism $\theta:H^{\bullet}(X)\to
H^{\bullet}(Y)$ is unique and identifies with the push-forward via
Poincar\'e Duality:
$$
H^{\bullet}(X)\cong H_{2(m+1)-\bullet}(X)\to
H_{2(m+1)-\bullet}(Y)\cong H^{\bullet}(Y).
$$
In fact $\theta={(k^*)}^{-1}\circ j^*$, and $k^*:H^h(Y)\to
H^h(U)$, for $h<2(m+1)$, is nothing but the duality morphism
because $H^h(U)\cong H_{2(m+1)-h}(Y)$.
\end{remark}


\begin{thebibliography}{s=5}

\bibitem{Br} Bredon, G. E.: {\it Sheaf Theory}, McGraw-Hill, New York 1967.

\bibitem{DeCM} de Cataldo, M.A. - Migliorini, L.: {\it The Gysin map is compatible with Mixed Hodge structures},
Algebraic structures and moduli spaces, 133-138, CRM Proc. Lecture Notes, 38, Amer. Math. Soc., Providence, RI, 2004.


\bibitem{DeCM2} de Cataldo, M.A. - Migliorini, L.: {\it The decomposition theorem, perverse sheaves and the topology
of algebraic maps}, Bulletin (new series) of the Amer. Math. Soc.,
535-633, Volume 46 (4), 2009.


\bibitem{DeCM3} de Cataldo, M.A. - Migliorini, L.: {\it Intersection forms, topology of maps and motivic decomposition
for resolutions of threefolds}, in Algebraic Cycles and Motives,
Cambridge, Cambridge University Press, 2007, pp. 102-137.


\bibitem{Del68} Deligne, P.: {\it Th\'eor\`eme de Lefschetz et crit\`eres de d\'eg\' en\' erescence de suites spectrales},
Inst. Hautes \'Etudes Sci. Publ. Math., No. 35, (1968), 259-278.

\bibitem{DGF1} Di Gennaro, V. - Franco, D.: {\it Factoriality and N\'eron-Severi groups},
Commun.  Contemp. Math. {\bf 10}, 745-764, 2008.

\bibitem{DGF} Di Gennaro, V. - Franco, D.: {\it Monodromy of a family of hypersurfaces},
Ann. Scient. \'Ec. Norm. Sup., $4^{e}$ s\'erie, t. 42, 517-529,
2009.


\bibitem{IJM} Di Gennaro, V. - Franco, D.: {\it Noether-Lefschetz Theory and N\'eron-Severi group},
Int. J. Math. {\bf 23} (2012), 1250004.


\bibitem{RCMP} Di Gennaro, V. - Franco, D.: {\it Noether-Lefschetz Theory with base locus},
Rend. Circ. Mat.  Palermo {\bf 63},  257-276, 2014.

\bibitem{DGFCCM2} Di Gennaro, V. - Franco, D.: {\it N\'eron-Severi group of a general hyperurface},
arxiv.org, 1506.07426v2,  1-14, 2016, to appear on Communications in Contemporary Mathematics.

\bibitem{Dimca1} Dimca, A.: {\it Singularities and Topology of Hypersurfaces}, Springer Universitext,
New York 1992.

\bibitem{Dimca2} Dimca, A.: {\it Sheaves in Topology}, Springer Universitext, 2004.



\bibitem{FultonIT} Fulton, W.: {\it Intersection theory}, Ergebnisse
der Mathematik und ihrer Grenzgebiete; 3.Folge, Bd. 2,
Springer-Verlag 1984.

\bibitem{FultonYT} Fulton, W.:{\it  Young Tableaux}, London Mathematical Society Student Texts 35.
Cambridge University Press 1997.


\bibitem{FultonCF} Fulton, W. - MacPherson R.: {\it Categorical framework for the study of singular spaces},
Mem. Amer. Math. Soc. 31 (1981), no. 243, pp. vi+165.


\bibitem{FrancoLascu} Franco, D. - Lascu, A.T.: {\it Curves contractable in general surfaces},
 Commutative algebra and algebraic geometry (Ferrara), 93-116, Lecture Notes in Pure and Appl. Math., 206, Dekker, New York, 1999.





\bibitem{Iversen} Iversen, B: {\it Cohomology of Sheaves},
Universitext, Springer, 1986.

\bibitem{Massey}  Massey, D. B.: {\it Intersection cohomology, monodromy and the Milnor fiber},
Internat. J. Math. 20, no. 4 (2009), 491-507.


\bibitem{McCrory}  McCrory, C.: {\it A characterization of homology manifolds},
J. London Math. Soc. (2), 16 (1977), 149-159.



\bibitem{Sacha1}  Lascu, A.: {\it Sous-vari\'et\'es contractibles d'une vari\'et\'e
alg\'ebrique}, Rend. Mat. (6) {\bf{1}} 1968 190-201.




\bibitem{Sacha2}  Lascu, A.: {\it On the contractability criterion of Castelnuovo-Enriques},
Atti Accad. Naz. Lincei Rend. Cl. Sci. Fis. Mat. Natur. (8)
{\bf{40}} 1966 1014-1019.

\bibitem{Spanier} Spanier, E.H.: {\it Algebraic Topology}, McGraw-Hill
Series in Higher Mathematics, 1966.

\bibitem{Verdier}  Verdier, J.L.: {\it Le th\'eor\`eme de Riemann-Roch pour les vari\'et\'es
alg\'ebriques \'eventuellement singuli\'eres (d'apr\'es P. Baum, W. Fulton et R. MacPherson)},
S\'eminaire de G\'eom\'etrie Analytique (\'Ecole Norm. Sup., Paris, 1974-75),
pp. 5-20 Asterisque, No. 36-37, Soc. Math. France, Paris, 1976.



\bibitem{Voisin} Voisin, C.: {\it Hodge Theory and Complex Algebraic Geometry, I},
Cambridge Studies in Advanced Mathematics 76.
 Cambridge University Press, 2007.
\end{thebibliography}
\end{document}